\documentclass[11pt,a4paper]{article}
\usepackage{soul}
\usepackage{amsmath,amsthm,amssymb,color,a4wide,stmaryrd}
\usepackage[all]{xypic}
\usepackage{xcolor}
\usepackage{MnSymbol}
\usepackage{graphicx}
\usepackage[normalem]{ulem}
\newlength{\bibitemsep}
\newlength{\bibparskip}
\let\oldthebibliography\thebibliography
\renewcommand\thebibliography[1]{%
  \oldthebibliography{#1}%
  \setlength{\parskip}{\bibitemsep}%
  \setlength{\itemsep}{\bibparskip}%
}
\usepackage{hyperref}
\hypersetup{
    colorlinks=true, 
    linktoc=all,     
    linkcolor=blue,  
}

\let\brick\relax
\let\blue\relax
 
\newcommand{\qee}{\mbox{\hspace{0.2mm}}\hfill$\triangle$}
\newcommand{\Z}{\mathbb{Z}}
\newcommand{\R}{\mathbb{R}}
\newcommand{\C}{\mathbb{C}}
\newcommand{\Q}{\mathbb{Q}}
\newcommand{\Pj}{\mathbb{P}}
\newcommand{\A}{\mathbb{A}}

\newcommand{\cO}{\mathcal{O}}

\newcommand{\K}{\mathbb{K}}

\newcommand{\Hom}{\operatorname{Hom}}

\newcommand{\Cl}{\operatorname{Cl}}

\newcommand{\Ann}{\operatorname{Ann}}
\newcommand{\eq}{=}

\DeclareMathOperator{\soc}{soc}
\DeclareMathOperator{\Hess}{Hess}

\newtheorem{thm}{Theorem}[section]
\newtheorem{prop}[thm]{Proposition}

\newtheorem{cor}[thm]{Corollary}
\newtheorem{conj}[thm]{Conjecture}
\newtheorem{dfn}[thm]{Definition}

\theoremstyle{remark}
\newtheorem{rmkk}[thm]{Remark}
\newtheorem{exe}[thm]{Example}
\newenvironment{rmk}{\begin{rmkk}\rm}{\qee\end{rmkk}}
\newenvironment{ex}{\begin{exe}\rm}{\qee\end{exe}}
\newcommand{\T}{\mathbb{T}}

\renewcommand{\t}{\mathfrak{t}}

\begin{document}
\title{\bf \large COX-GORENSTEIN ALGEBRAS } 
\author{\sc\normalsize Ugo Bruzzo,$^a$  Rodrigo Gondim,$^b$ Rafael Holanda$^c$ and  William D.  Montoya$^d$}
\date{\small $^a$ 
Departamento de Matem\'atica, Instituto de Ci\^encias Exatas,  Universidade Federal \\ de Minas Gerais, 
Av.~Ant\^onio Carlos 6627,   Belo Horizonte, MG 30123‑970, Brazil
 \\  
$^a$ INFN (Istituto Nazionale di Fisica Nucleare), Sezione di Trieste  \\ $^a$ IGAP (Institute for Geometry and Physics), Trieste \\
$^b$ Universidade Federal Rural de Pernambuco, Av. Don Manoel de Medeiros
s/n, \\ Dois Irm\~aos, Recife, PE 52171-900, Brazil \\
$^c$ Departamento de Matem\'atica, CCEN, Universidade Federal de Pernambuco,\\ Rua Jornalista An\'ibal Fernandes, Cidade Universit\'aria, Recife, PE 50740-560, Brazil \\
$^d$ Instituto de Matem\'atica,  Estat\'istica e Computa\c c\~ao  Cient\'ifica,   Universidad Estadual de Campinas, Rua S\'ergio Buarque  de Holanda 651, 
Campinas, SP  13083-859, Brazil 
}
\maketitle
\begin{abstract} We study   $G$-graded Artinian algebras having Poincar\'e duality, considering in particular their Lefschetz properties. We also prove a correspondence between the toric setup and the $G$-graded one, provide an application to toric geometry, and  prove a Hessian criterion in the $G$-graded setup. 
\end{abstract}
\let\svthefootnote\thefootnote
\let\thefootnote\relax\footnote{
\hskip-2\parindent 
Date:  5 November 2025  
Email: {\tt ubruzzo@ufmg.br, rodrigo.gondim@ufrpe.br, rafael.holanda@ufpe.br,\\  \parbox{26pt}{\hfill} wmontoya@ime.unicamp.br} \\
\brick{U.B.: Research partly supported  by Bolsa de Produtividade 305343/2025-4 of Brazilian CNPq,  by GNSAGA-INdAM and by the PRIN project 2022BTA242 ``Geometry of algebraic structures: moduli, invariants, deformations.'' \\
R.G.: Research partially supported by the CNPq Research Fellowship  309344/2023-9.\\
R.H.: Research partially supported by the CNPq grants 200863/2022-3 and 170235/2023-8.\\
W.M.: Research supported by FAPESP postdoctoral grants   2019/23499-7 and 2023/01360-2. }
}

\addtocounter{footnote}{-1}\let\thefootnote\svthefootnote

\newpage

\tableofcontents

\newpage

\section{Introduction}

Artinian Gorenstein $\K$-algebras are an algebraic model for cohomology rings in several situations; the reason for this is  Poincaré duality, which works in these algebras when they are commutative and standard graded. 
In many geometric situations, especially those involving toric varieties, non-standard graded algebras arise --- for example, in the study of the Hodge structure of the primitive cohomology ring   (\cite{BatyrevCox}, \cite{Mavlyutov}), in the Noether-Lefschetz loci (\cite{BruzzoMontoya2021}, \cite{BruzzoMontoya2025}, \cite{Montoya}) and in the theory of periods of complete intersections, in particular hypersurfaces in toric varieties \cite{Villaflor}.

On the other hand, $G$-graded polynomial rings,
where $G$ is an Abelian group,  appear naturally as Cox rings of toric varieties. Given a partially ordered Abelian group $(G, \preceq)$, we want to study Artinian Gorenstein $G$-graded $\K$-algebras, which we will call Cox-Gorenstein algebras. One of the pioneering works on $G$-graded algebras, with $G$ a finitely generated free abelian group, is \cite{goto1978graded}.

Given a complete normal toric variety $\Pj_\Sigma$, its class group $\Cl(\Pj_\Sigma)$ is a finitely generated abelian group which has a natural partial order making effective divisors positive. 
Moreover, every finitely generated abelian group can be presented up to isomorphism as the class group of some toric variety. Conversely, given a finitely generated partially ordered abelian group $G$ and a $G$-grading over a polynomial ring $S$, we want to recover the toric variety when the rank of $G$ is less than or equal to the number of variables of $S$. In Section \ref{$G$-graded polynomial rings and toric varieties} we give an algorithm proving this correspondence.

Let $(G, \preceq)$ be a partially ordered group. From the established correspondence, we consider $G$-graded $\K$-algebras with $\K$ a field of characteristic zero $A=\oplus_{g \in G} A_g$. In Section \ref{Toric Macaulay-Matlis duality} we investigate the $G$-graded version of the Macaulay-Matlis duality. In that section, we introduce the Hasse-Hilbert diagram for algebras that plug the Hilbert function of $A$ in its Hasse diagram given by the partial order of $(G, \preceq)$. Cox-Gorenstein algebras have finite symmetric Hasse-Hilbert diagram having a greatest element. In the \brick{$G$-graded} case, there are examples of non-Artinian algebras whose grading group has maximal elements, see Example \ref{ex4}. The geometric reason for this issue is that these algebras come from toric varieties with   Picard number greater than one, see Theorem \ref{picard1} and Example \ref{counterexample}. When an algebra has a maximal element, we can construct a minimal Artinian Gorenstein quotient that preserves the pairing up to the greatest degree, see Proposition \ref{prop:artinianization} and \ref{prop:gorenstenianization}. As a consequence, the primitive cohomology of a hypersurface, which captures the part of the cohomology that does not come from the ambient toric variety, embeds naturally into a Cox-Gorenstein algebra.

The Lefschetz properties for standard-graded Artinian algebras are algebraic abstractions of the so-called hard Lefschetz theorem on the cohomology of smooth projective complex varieties. In the last decades, similar results have been obtained in several categories. From the algebraic perspective, Lefschetz properties configure an active area of research and most of the results concern standard graded algebras. In \brick{ \cite{HarimaWatanabe,Iarrobino-etal,Mcdaniel-etal} strong Lefschetz property was studied for Artinian algebras with nonstandard grading; where nonstandard means that the grading group is $\Z$, allowing different degrees for the variables.} In Section \ref{lefschetz properties}, we introduce \brick{a Lefschetz property} for \brick{$G$-graded} Artinian algebras, where $G$ can be any finitely generated abelian group. The Lefschetz properties have a deep influence in the Hilbert function in the classical case, it is one of its main features. On the other hand, the theory is wide open in the toric setting. In addition to the natural problems arising from the standard case, in the toric case, Lefschetz properties are also related with the concept of Oda variety (see \cite{BruzzoGrassi2}).

 Macaulay-Matlis duality yields a one-to-one correspondence between Artinian quotients of a polynomial ring of differential operators and a finitely generated module over a polynomial ring. Under this correspondence, an Artinian Gorenstein algebra can be represented by the quotient of a ring of differential operators modulo the annihilator of a single polynomial. This is a well-known result in the standard case, and in Section \ref{Macaulay_Matlis},  we restate it for the \brick{$G$-graded case, see also \cite{KleimanKleppe} for a more general construction}. In \cite{MaenoWatanabe2009}, the authors proved that the Strong Lefschetz properties in this case can be controlled by the rank of certain higher-order Hessians; this result was generalized in \cite{GondimZappala'2018}, where the authors proved that the matrix of any multiplication map from the power of a linear form is a mixed Hessian. In Section \ref{lefschetz properties}, we generalize the so-called Hessian criteria to the toric case.

In Section \ref{A geometric case for Cox-Gorenstein algebras} we consider a special case of the codimension one conjecture, posed by Cattani-Cox-Dickenstein in \cite{1997}. First of all, we would like to stress that this conjecture can be viewed as a toric version of part of the classical theorem by Macaulay stating that an Artinian set-theoretically complete intersection in a polynomial ring is a schematic complete intersection. Thus, if $A$ stands for one such Artinian complete intersection, it must be Gorenstein and then $\dim A_{\omega} =1$, where $\omega$ is its socle degree. The first incarnation of such kind of results for toric varieties is in \cite{COX}. We make use of spectral sequences to generalize the Residue Isomorphism Theorem \cite[Theorem 0.3]{1997} for polynomials of suitable numerically effective degrees, which in particular include $\mathbb Q$-ample divisors, see Theorem \ref{picard1}. Our result also implies essentially \cite[Theorem 1.3]{Villaflor} while in \cite{Cox-Dick} the authors proved a special case of the conjecture that actually implies our result (see  \cite[Corollary 2.5]{Cox-Dick}). As described above, the assumption of Picard number one is essential, leading us to think of how to produce Cox-Gorenstein algebras preserving the Hasse-Hilbert diagram.  
 
\medskip
\noindent {\bf Acknowledgments.} We thank Alicia Dickenstein for letting us know about reference \cite{Cox-Dick}, Tony Iarrobino for exhaustively informing us about previous work on nonstandardly graded Gorenstein algebras, and Antonio Laface for useful comments. R.G.~appreciates the hospitality of Universit\`a di Catania, UNICAMP and Universidade Federal da Para\'iba, where part of this work was developed. R.H.~joined this project while visiting the first author at SISSA in Trieste; he is very appreciative of the hospitality he received there.    W.M.~appreciates the hospitality of the  University of Ferrara, University of Catania, SISSA, and Universidade Federal Rural de Pernambuco where part of this work was developed.

\section{Preliminaries: $G$-graded algebras and toric varieties}

\subsection{$G$-graded algebras}

Let $(G,+,\preceq)$ be a partially ordered Abelian group, and consider the monoid $G_+=\{g \in G\ |\ g\succeq 0\}$. Let $\K$ be a field of characteristic zero. 

\begin{dfn}\rm

Let $R$ be a $\K$-algebra. A $G$-grading in $R$ is a decomposition of $R$ into $\K$-vector spaces $R=\oplus_{g \in G}R_g$ such that the product in $R$ satisfies $R_gR_h \subset R_{g+h}$. We always assume $R_0=\K$ and we say that $R$ is finitely graded if $R_g=0$ for all but finitely many $g\in G$ and positively graded if $R_g = 0$ for all $g \not \in G_+$, that is, $R$ is $G_+$-graded. Given a homogeneous element $f \in R_g\setminus 0$, we denote $\deg(f)=g\in G$.  A $G$-graded $R$-module is an $R$-module $M$ having a decomposition $M=\oplus_{g \in G}M_g$ as $\K$-vector spaces satisfying $R_gM_h \subset M_{g+h}$. A $G$-homogeneous ideal $I \subset R$ is a $G$-graded $R$-submodule of $R$.
    
\end{dfn}

Following \cite{gordon1982graded}, we say that a $G$-graded $\K$-algebra $A$ is Artinian if, neglecting the grading, it is an Artinian ring.

\subsection{Toric varieties}

\begin{dfn} \rm A toric variety is a normal irreducible variety $X$ containing an algebraic torus $T \simeq (\C^*)^d$ as a Zariski open subset such that the action $T\times T\rightarrow T$ of $T$ on itself extends to an algebraic action of $T$ on $X$.  
    
\end{dfn} 

We summarize some basic definitions and properties of cones and fans related to toric varieties.

\begin{dfn}
Let $M$ be a free abelian group of rank $n$. Let
$$N = \Hom_{\Z}(M, \Z), \qquad N_{\R} = N \otimes_{\Z} \R.$$
 \begin{itemize}
     \item A convex subset $\sigma \subset N_{R}$ is a rational $s$-dimensional  cone if there exist over $\R$,  $s$-elements
 $e_1,\dots, e_s \in N$ such that $$\sigma = \{\mu_1e_1 +\cdots+\mu_se_s \mid (\mu_1,\dots,\mu_s) \in \R_+^s\}.$$
     \item For any $i \in \{1,\ldots,s\}$, the generator $e_i$ is integral if for any non-negative rational number $l$ the product $l. e_i$ is in $N$ only if $l$ is an integer.
     \item Given two rational  cones $\sigma$, $\sigma'$ one says that $\sigma' $ is a face of $\sigma$ ($\sigma'< \sigma$)
     if the set of integral generators of $\sigma'$ is a subset of the set of integral generators of $\sigma$.
     \item A cone $\sigma$ is strongly convex if $\{0\}$ is a  face of $\sigma$.
     \item  A finite set $\Sigma = \{\sigma_1,\dots, \sigma_t\}$ of strongly convex rational  cones is called a fan if:
\begin{itemize}
    \item all faces of cones in $\Sigma$ are in $\Sigma$;
    \item if $\sigma,\sigma'\in \Sigma$ then $\sigma\cap \sigma'< \sigma$ and $\sigma\cap \sigma' <\sigma'$.
\end{itemize}     
\end{itemize}
\end{dfn}

A fan $\Sigma\subset N_{\R}$ defines a \textit{toric variety} $X_\Sigma$ with torus $T_N=N\otimes_{\Z} \C^{*}$. A strongly convex rational polyhedral cone $\sigma\subset N_{\R}$ is simplicial if its minimal generators are linearly independent over $\R$ and we say that a fan $\Sigma$ is simplicial if every cone $\sigma$ in $\Sigma$ is simplicial.

\begin{dfn}\rm  Let $X$ be a complete normal variety with finitely generated class group. The {Cox ring} of $X$ is the graded ring 
$${\rm Cox}(X):=\bigoplus_{[D]\in {\rm Cl}(X)}H^0(\cO_X(D))$$
    
\end{dfn}

We refer to \cite{ArzhantsevDerenthalHausenLaface} for an exhaustive study of this ring.

\begin{prop} Let $X$ be an irreducible normal variety with only constant invertible global functions and finitely generated divisor class group such that the Cox ring of $X$ is a polynomial ring. If $X$ is complete, then $X$ is a toric variety.
    
\end{prop}

\begin{proof} Ex. 3.9 in \cite{ArzhantsevDerenthalHausenLaface}.
    
\end{proof}

Given a fan $\Sigma$ we call $\Sigma(1)$ the set of rays of $\Sigma$, each $\varrho_i\in \Sigma(1)$ corresponds to an irreducible $T$-invariant Weil divisor $D_{\varrho_i}$ on $X_\Sigma$. Consider a variable $x_i$ for each  $\varrho_i\in \Sigma(1)$ and let $S$ be the Cox ring of $X_{\Sigma}$, i.e., the polynomial ring $\C[x_1,\dots,x_n]$ with the $\Cl(X_\Sigma)$-grading which is given as follows. A monomial $x^a:=\prod_{\varrho\in\Sigma(1)}x_{\varrho}^{a_{\varrho}}\in S   $ is associated to the Weil divisor $D=\sum_{\varrho \in \Sigma(1)} a_\varrho D_\varrho$. Then $$\deg(x^a)=[D]\in \Cl(X_\Sigma).$$
For every $\sigma\in \Sigma$, let $x_{\sigma}=\prod_{\rho_i\not\subset \sigma} x_i$. The ideal in $S$ generated by the $x_{\sigma}'s$ is called the irrelevant ideal and we denote it by $B(\Sigma)$.

\section{$G$-graded polynomial rings and toric varieties}
\label{$G$-graded polynomial rings and toric varieties}

 In this section, we describe the following construction. Fix a positive integer $n$,  a finitely generated abelian group $G$ of rank $\rho$, 
and a grading of the ring $\K[x_1,\dots,x_n]$, where $\K$ is any field, by the group $G$.
The algorithm produces a monomial ideal $I$ in $\K[x_1,\dots,x_n]$. The algorithm uses toric geometry (over $\C$), building a toric variety $X$ of dimension $n-\rho$ whose class group is   $G$, and the ideal $I$ is the irrelevant ideal of the Cox ring of $X$. Since $I$ is a monomial ideal, it may be regarded as an ideal in any polynomial ring.

We shall assume the following two conditions (see \cite[Exercise 2.1]{ArzhantsevDerenthalHausenLaface}):
\begin{enumerate} \itemsep=-2pt
\item If $x_i$ has degree $w_i\in G$, then for every $i=1,\dots,n$, the elements $w_1,\dots,w_{i-1},w_{i+1},\dots, w_n$ generate $G$;
\item for every $i,j=1,\dots,n$, $i\ne j$, one has
$$ \operatorname{cone}(w_k\,\vert\,k\ne i)^\circ \cap \operatorname{cone}(w_k\,\vert\,k\ne j)^\circ \ne \emptyset $$
\end{enumerate}
(cones are in $G\otimes_\Z\R$). Note that the first condition implies that $\rho<n$.

Fix   integers $\alpha_1,\dots,\alpha_n$ and define an action of the $n$-dimensional complex algebraic torus $\T^n$ on $R=\C[x_1,\dots,x_n,t] $
by letting 
$$ \lambda (x_1,\dots,x_n,t) = (\lambda_1x_1,\dots,\lambda_nx_n,\lambda_1^{\alpha_1}\cdots\lambda_n^{\alpha_n}t)$$
if $\lambda=(\lambda_1,\dots,\lambda_n)\in \T^n$.
Regarding $R$ as the affine coordinate ring of the trivial bundle on $\C^n$, this is a linearization on the trivial bundle of the standard action of $\T^n$ on $\C^n$.
Let $\Gamma$ be a subgroup of  $\T^n$, and consider the exact sequence
$$ 1 \to \Gamma \to \T^n \to \T \to 1;$$
let
\begin{equation} \label{Lie} 0 \to \pmb\gamma \to \t^n \xrightarrow{\ \pi \ } \t \to 0  \end{equation}
be the induced exact sequence of Lie algebras. $\t$ has an embedded lattice $$\t_\Z = \ker (\exp \colon \t \to \T)$$ and therefore
also a real section $\t_\R = \t_\Z\otimes \R$.  Note that being a subgroup of $\T^n$, the group $\Gamma$ acts on $\C^n$ and this action has a linearization on the trivial line bundle.

Since $\T^n = (\C^\ast)^n$ there is an identification of $\t^n$ with $\C^n$ so that
we may consider   basis vectors $\{e_i\}$ in $\t^n$. 
Let $a_i = \pi(e_i)$, and define the polyhedron
\begin{equation}\label{poly} \Delta = \{ \xi \in \t_\R^\ast =  \t_\Z^\ast\otimes \R \, \vert \, \xi(a_i) \ge \alpha_i\}.  \end{equation} 
Note that the facets of $\Delta$ have a natural numbering from 1 to $n$.
We have the following facts:
\begin{enumerate} \itemsep=-2pt
\item  $\Delta$ is the polyhedron associated to a pair $(X,D)$, where $X$ is a toric variety, and 
$D$ is a divisor in it given by $$ D = -  \sum_i \alpha_i \, D_i$$
where the $D_i$ are the divisors corresponding to the rays of the fan dual to $\Delta$;
\item $X$ is the GIT quotient $\C^n /\!/ \Gamma$;
\item The class group of $X$ is $\Cl(X) \simeq \Hom(\Gamma,\C^\ast)$.
\end{enumerate}
. 

To start with, we assume that $\Gamma$ is torsion-free, i.e., it is an algebraic torus. Let $\rho$ be its dimension (geometrically, the Picard number of $X$).  The embedding $\C^\rho \hookrightarrow \C^n$ is described by an $n \times \rho$ matrix of integers (the weights of the action of $\Gamma$)
$$ \{P_{ij}\}, \qquad i=1,\dots,n, \ j = 1,\dots,\rho  $$
 Note that each column of the matrix $P$  defines a character of $\Gamma$, i.e., it may be regarded as an element in $G=\Gamma^\vee = \Hom_\Z(\Gamma,\C^\ast) \simeq \Z^\rho$.
The Cox ring $S(X)$ of $X$ is the polynomial ring $\C[x_1,\dots,x_n]$ with the grading given by the matrix $P$.

\begin{rmk}
  The dual fan of $\Delta$ is independent of the choice of the numbers $\alpha_i$, and therefore ultimately the irrelevant ideal only depends on the embedding of $\Gamma$ into $\T^n$, that is, a posteriori, on the grading of the Cox ring. 
\end{rmk}

\begin{rmk}
The irrelevant locus (which is the unstable locus) may be bigger than the locus where the action of $\Gamma$ is not free.
\end{rmk}

\begin{ex} \label{F1} Fix $n=4$ and $\alpha_i = - 1$. The group $\Gamma$ is $\T^2=\C^\ast\times\C^\ast$ embedded into $\T^4$ as described by the weight matrix
$$ P = \begin{pmatrix} 1 & 0  \\1 & 0 \\ 0 & 1 \\ -1 & 1 \end{pmatrix} $$
The sequence \eqref{Lie} is
$$ 0 \to \C^2 \xrightarrow{\ \iota\ } \C^4 \xrightarrow{\ \pi\ } \C^2 \to 0 $$
with
$$ \iota(x,y) = (x,x,y,-x+y), \qquad \pi(u,v,s,t) = (-u+s-t,u-v)$$
so that
\begin{equation}\label{ai} a_1=(-1,1), \quad a_2 = (0,-1), \quad a_3 = (1,0), \quad a_4 = (-1,0). \end{equation}
The vectors $a_i$ are the rational minimal generators of the rays of the fan.
The irrelevant ideal is
$$ I = (tv, vs,su,ut).$$
The degrees of the variables are
$$\deg u = \deg v = (1,0), \quad \deg s = (0,1), \quad \deg t = (-1,1).$$
Clearly, the toric variety $X$ is the first Hirzebruch surface. Due to our choice $\alpha_i = -1$, the divisor $D$ is the sum of the toric-invariant divisors, i.e., it is an anti-canonical divisor.
\end{ex}

  \begin{ex} \label{fake} 
  We consider an example with torsion.  This time we start from the fan, and consider   a fake projective plane described by the fan in $\R^2 = \Z^2 \otimes \R$ with rays generated by 
  \begin{equation} v_1=(2,-1), \quad v_2(-1,2), \quad v_3=(-1,-1). \label{fakerays} \end{equation}
  This is a quotient $\Pj^2/\Z_3$ and actually is the cubic surface $w^3=xyz$ in $\Pj^3$. The class group is $\Z\oplus\Z_3$, so that $\Gamma = \C^\ast\times \mu_3$, where $\mu_3$ is the group of cubic roots of 1. If $\omega$ is a primitive cubic root of 1 the monomorphism $\Gamma \to \T^3$ is given by
\begin{equation}  (\lambda,1)  \mapsto  ( \lambda, \lambda, \lambda ), \quad  (1,\omega) \mapsto (1,\omega,\omega^2). \label{fakeweights} \end{equation}
The map $\pi \colon \C^3 \to \C^2$ is
  $$\pi(z_1,z_2,z_3) = (2z_1-z_2-z_3, -z_1+2z_2 - z_3).$$
  The $a_i$ are
\begin{equation} \label{rec-ai} a_1=(2,-1), \quad a_2=(-1,2), \quad a_3=(-1,-1). \end{equation}
  that is, we recover the generators of the rays in equation \eqref{fakerays}, as it must be.
  
  The grading group is the class group $\Cl(X)=\Hom_\Z( \C^\ast\times \Z_3,\Z) = \Z\oplus\Z_3$. Equation \eqref{fakeweights} suggests that in this case instead of a matrix $P$ we have two matrices 
  (actually, column vectors) $P_0$, $P_1$, providing the weights of the actions of $\C^\ast$ and $\Z_3$, respectively:\footnote{To simplify the notation we write the matrices $P$ as row vectors.}
  \begin{equation} P_0=  (1,1,1), \quad  P_1 = (0,1,2).\label{fakeweights2} \end{equation} 
  The Cox ring is $\C[x_1,x_2,x_3]$ and the degrees of the variables are
  $$ \deg x_1 = (1,\overline 0), \quad \deg x_2 = (1,\overline 1), \quad \deg x_3 = (1,\overline 2).$$
  The irrelevant ideal is the maximal ideal --- as it happens for all (weighted)  projective spaces, fake or not. 
    \end{ex}
   
     \begin{ex} \label{fake2} We study another example with torsion --- another fake projective space \cite[Example 3.2(3)]{BruzzoMontoya2021}. The  fan in $\R^2 = \Z^2 \otimes \R$ has  rays generated by 
  \begin{equation} v_1=(-3,-2), \quad v_2=(1,2), \quad v_3=(1,0). \label{fake2rays} \end{equation} 
  We have 
  $$ G = \Cl(X) = \Z \oplus \Z_2, \qquad \Gamma = \C^\star \times \mu_2.$$
  Developing the same computations as before we have that the monomorphism 
  $\Gamma \to \T^3$ is given by
\begin{equation}  (\lambda,1)  \mapsto  ( \lambda, \lambda, \lambda^2 ), \quad  (1,-1) \mapsto (1,-1,-1)\label{fake2weights} \end{equation}
(here we see explicitly that $X \simeq \Pj[1,1,2]/\Z_2$).  The matrices $P$ are
\begin{equation}\label{Pfake2} P_0 = (1,1,2), \qquad P_1=(0,1,1).\end{equation} The degrees of the variables are
$$\deg x_1 = (1,\overline 0), \quad \deg x_2 =(1,\overline 1), \quad \deg x_3 = (2,\overline 1).$$
   \end{ex}   
 
  \begin{rmk} The construction works also in the noncomplete case; we just get an open polyhedron instead of a polytope, and the fan is not complete. 
  For instance for the fan $v_1=(0,1)$, $v_2=(2,-1)$ (whose associated toric variety is the cone $z^2 = xy$ in $\mathbb A^3$) we have the polyhedron determined  by the inequalities $ y \ge -1$, $ y \le 2x+1$. The class group  $G$  is $\Z_2$, so that 
  $$\Gamma = \Hom_\Z(\Z_2,\C^\ast)  \simeq  \mu_2,$$
  the group of square roots of 1.   
 One has one matrix $P=(1,1)$, and the Cox ring is $\C[x_1,x_2]$ with grading
  $$\deg x_1 =  \deg x_2 = \overline 1 \in  \Z_2.$$
  \end{rmk}

  Now we describe the promised algorithm; the previous examples make it clear that we can proceed in the following way. We allow for the presence of torsion. 
  \begin{itemize}\itemsep = -2pt
 \item Consider the polynomial ring $S = \C[x_1,\dots,x_n]$, graded by a   finitely generated abelian group $G$
$$G  = \Z^\rho \oplus \Z_{m_1} \oplus \cdots \oplus \Z_{m_N}.$$
We assume that the two conditions at the beginning of this section are satisfied.
 The grading of the variables  defines $N$ matrices   $P_i$, $i = 0,\dots,N$. Actually, $P_0$ is a $ n\times\rho $ matrix, and the $P_i$ with $i>0$ are
 column vectors with $n$ entries. 
 \item  Let $\Gamma=\Hom_\Z(G,\C^\ast)$ and embed it into $\T^n$ as prescribed by the matrices $P_i$.
 \item Following the prescription described above, one gets first a polytope, then a fan $\Sigma$, and from it the irrelevant ideal $I$. 
 \item We also get a toric variety $X$, given by the fan $\Sigma$; it is the GIT quotient
 $$ X = \mathbb A^n /\!/\,\Gamma = \frac{\A^n - V(I)}{\Gamma}
 .$$
\end{itemize}

\begin{ex} Let us use Example \ref{fake} to check how this construction works. The input data are the choice of $n=3$ and the grading of the ring $\K[x_1,x_2,x_3]$ by the group $G = \Z\oplus\Z_3$ given by the specification of the matrices $P$ in equation \eqref{Pfake2}. This gives the exact sequence \eqref{Lie} the form
$$ 0 \to \C \to \C^3 \xrightarrow{\ \pi\ } \C^2 \to 0 $$
with   $$\pi(z_1,z_2,z_3) = (2z_1-z_2-z_3, -z_1+2z_2 - z_3).$$
The prescription $a_i = \pi(e_i)$ determines the vectors $a_i$
  \begin{equation}  a_1=(2,-1), \quad a_2=(-1,2), \quad a_3=(-1,-1)\end{equation} 
  and taking these as the generators of the rays of a fan we recover the variety of the Example \ref{fake} (compare equation \eqref{fakerays}), with its irrelevant ideal $I = (x_1,x_2,x_3)$. 
\end{ex}

\section{Cox-Gorenstein algebras}

\label{Toric Macaulay-Matlis duality}

\subsection{Hasse-Hilbert diagrams}

We know that any   Artinian ($G$-graded) $\K$-algebra $A$ has finite dimension over $\K$, therefore it is finitely graded and every graded piece is a finite-dimensional $\K$-vector space; we put $h_g=\dim A_g$. The Hilbert function of $A$
\[\operatorname{HF}_A:G \to \Z_{+} \]
is defined as $\operatorname{HF}_A(g)=h_g$. To properly encode the structural information contained in the Hilbert function, we introduce the Hasse-Hilbert diagram of $A$ defined as a vertex-weighted directed graph structure over the covering graph of $G$ where the weight of a vertex $g$ is the Hilbert function $h_g$. By definition of the covering graph of a partial order, the vertex set is a POSET, in our case $G$, and two vertices $g,h\in G$ are connected if they are immediate neighbors, that is, they are comparable and there is no other comparable element between them. As usual, a maximal element in a POSET $X$ is an element $x\in X$ for which there does not exist $y\in X$ such that $x \preceq y$ with $x\neq y $. We say that a maximal element $x \in X$ is the greatest element in $X$ if $y \preceq x$ for all $y \in X$.   

\begin{ex}\label{ex1} 
    Let $S=\K[x,y,u,v]$ be $\Z^2$-graded by $\deg(x)=\deg(y)=(1,0)$ and $\deg(u)=\deg(v)=(0,1)$. We consider $\Z^2$ with the order $(a,b)\preceq (c,d)$ if and only if $a\leq c$ and $b\leq d$. 
Let $I = (S_{(0,2)}, S_{(5,0)}, x^2u-y^2v,x^2v,y^2u) \subset S$, the quotient $A=S/I$ is an Artinian $\Z^2$-graded algebra. The Hasse-Hilbert diagram has as the greatest element $A_{(4,1)}$. In fact $A=\oplus_{g \leq (4,1)}A_g$.  

$$
\xymatrix@=.4em{ 
& & {\begin{matrix}{\color{red}2}\\(1,0)\end{matrix}} \ar[rd]& & & &   {\begin{matrix}{\color{red}2}\\(3,1)\end{matrix}}\ar[rrdd]  & \\
& & &{\begin{matrix}{\color{red}3}&(2,0)\end{matrix}}\ar[rrdd]\ar[rr] & &{\begin{matrix}(2,1)&{\color{red}3}\end{matrix}}\ar[ru] & \\
{\begin{matrix}{\color{red}1}&(0,0)\end{matrix}}\ar[rruu]\ar[rrdd]&&&&&&&&{\begin{matrix}(4,1)&{\color{red}1}\end{matrix}}\\
&  & & {\begin{matrix}{\color{red}4}&(1,1)\end{matrix}}\ar[rruu]& &{\begin{matrix}(3,0)&{\color{red}4}\end{matrix}}\ar[rd] &&\\
& &{\begin{matrix}(0,1)\\{\color{red}2}\end{matrix}}\ar[ru] &&&&{\begin{matrix}(4,0)\\{\color{red}5}\end{matrix}}\ar[rruu]&
}
$$

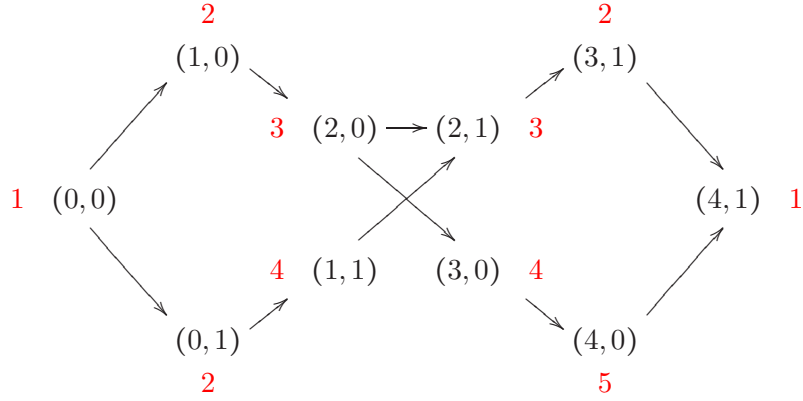
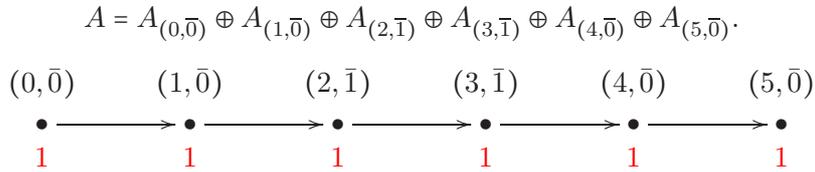
\centering{Hasse-Hilbert diagram of A}

\end{ex}

When the order is total the Hasse-Hilbert diagram is linear and we can write $A=\oplus_{i=1}^n A_{g_i}$ where $n=|G|$.

\begin{ex}\label{ex2} \rm  (See Example \ref{fake2}). 
    Let $S=\K[x,y,z]$ be the polynomial ring with a $G=\Z\oplus \Z_2$-grading given by $\deg(x) = (1,\overline{1})$, $\deg(y) = (1,\overline{0})$ and $\deg(z) = (2,\overline{1})$. We consider $G$ with the order $(a,\overline{b})\preceq (c,\overline{d})$ if and only if $a\leq c$. Let $I=(x,y^2,z^3)\subset S$ and $A=S/I$. It is easy to see that $A$ is an Artinian $G$-graded algebra. The Hasse-Hilbert diagram of $A$ is linear and we can write 
\[A=A_{(0,\overline{0})} \oplus A_{(1,\overline{0})} \oplus A_{(2,\overline{1})}\oplus A_{(3,\overline{1})} \oplus A_{(4,\overline{0})}\oplus A_{(5,\overline{0})}.\]

$$
\xymatrix@R=.1em{ 
(0,\bar{0}) & (1,\bar{0})& (2,\bar{1}) & (3,\bar{1})& (4,\bar{0}) &(5,\bar{0})\\
\bullet\ar[r] & \bullet \ar[r] & \bullet\ar[r] & \bullet\ar[r] & \bullet\ar[r] &\bullet\\
{\color{red}1}&{\color{red}1}&{\color{red}1}&{\color{red}1}&{\color{red}1}&{\color{red}1}
}
$$

\centering{Hasse-Hilbert diagram of A}

\end{ex}

The Hasse-Hilbert diagram can be defined in general for algebras of finite type, but it is finite if and only if $A$ is Artinian.

\begin{dfn} Let $A=\K[X_1,\ldots,X_n]/I=\oplus_{g \in G} A_g$ be an Artinian $G$-graded $\K$-algebra and let $\mathfrak{m}:=(\overline{X_1},\ldots,\overline{X_n})\subset A$. We say that $A$ is Cox-Gorenstein if there is $\omega\in G$
such that $\soc(A):=(0:\mathfrak{m}) = A_{\omega} \simeq \K$. In this case, $\omega$ is called the socle degree of $A$.      
\end{dfn}

Note that the ideal $\mathfrak{m}$ of the Artinian $G$-graded $\K$-algebra $A$ is the only homogeneous maximal ideal of $A$. Moreover, it turns out that $A$ is indeed a local ring with maximal ideal $\mathfrak{m}$. It allows us to invoke the following result in the \brick{$G$-graded} setting.

\begin{thm}{\rm(\cite[Theorem 2.58]{HarimaMaenoMoritaNumataWackiWatanabi})}\label{gorenstein} Let $A$ be an Artinian $G$-graded $\K$-algebra. The following are equivalent:
\begin{enumerate} \itemsep=-2pt
    \item $A$ is Cox-Gorenstein. 
    \item The ideal $(0)$ is irreducible.
    \item $A$ has finite self-injective dimension.
    \item $A$ has a unique minimal non-zero ideal. 
\end{enumerate}
\end{thm}

\begin{dfn} Let $A=\oplus_{g \in G} A_g$ be an Artinian $G$-graded $\K$-algebra having a greatest element $A_{\omega}$. 
We say that $A$ has the Poincar\'e duality if $A_{\omega}\simeq\K$ and the multiplication maps 
\[A_{g}\times A_{\omega-g} \rightarrow A_{\omega}\] are perfect pairings, whenever $A_{g}$ and $A_{\omega-g}$ are non-zero.  
Moreover, if one is zero, the other is also zero.     \label{PDdefinition}
\end{dfn}

\begin{thm} Let $A=\oplus_{g \in G} A_g$ be an Artinian $G$-graded $\K$-algebra. 
$A$ is Cox-Gorenstein if and only if $A$ has the Poincar\'e duality.
     \label{gorPD} 
\end{thm} 
\begin{proof} 
If $A=\oplus_{g \in G} A_g$ has the Poincaré duality then $\operatorname{soc}(A)\subseteq A_{\omega}$. Since $\operatorname{soc}(A)$ is a non trivial $\K$-vector space contained in the $1$-dimensional $\K$-vector space $A_{\omega}$,   they are equal.
For the other implication let $0\neq f\in A_{g}$ and let $J$ be the intersection of all non-zero ideals of $A$. Then by Theorem \ref{gorenstein} part (iv) $J\neq (0)$.  Now one notes that $J=\soc(A)=(h)$ for some $h$ and since $f\neq 0$ we have that 
$J\subseteq  (f)$ and hence there exists $f'\in A_{\omega-g}$ such that $f\cdot f'=h$.
    \end{proof}

\begin{rmkk}\rm
Notice that the Artinian algebra in Example \ref{ex2} is Cox-Gorenstein while the one given in Example \ref{ex1} is not. The Hasse-Hilbert diagram of a Cox-Gorenstein algebra is symmetric and has a greatest element; the converse is not true.     

$$
\xymatrix@=.9em{
& \bullet\ar[r]\ar[rdd]&  \bullet\ar[r]\ar[r] &\cdots \ar[r]& \bullet\ar[r]\ar[rdd] &  \bullet\ar[rd]&  \\
\bullet\ar[ru]\ar[rd] &   & & & & &\bullet  \\ 
& \bullet\ar[r]\ar[ruu] &\ar[r]   \bullet &\cdots\ar[r]  &\bullet\ar[r]\ar[ruu]  & \bullet\ar[ru] &
}
$$
\centering{Hasse-Hilbert diagram of a Cox-Gorenstein algebra}

\end{rmkk}

\begin{ex}\rm 
(See Example \ref{ex1}).
 Let $S=\K[x,y,u,v]$ be $\Z^2$-graded by $\deg(x)=\deg(y)=(1,0)$ and $\deg(u)=\deg(v)=(0,1)$. We consider $\Z^2$ with the order $(a,b)\preceq (c,d)$ if and only if $a\leq c$ and $b\leq d$ as in Example \ref{ex1}. 
Let $I = (S_{(0,2)}, S_{(5,0)}, x^2u-y^2v,x^2v,y^2u ) \subset S$, the quotient $A=S/I$ is an Artinian $\Z^2$-graded algebra (see Example \ref{ex1}). The socle of $A$ has a $3$ dimensional subspace in $A_{(4,0)}$ which is 
$\soc(A)\cap A_{(4,0)} = \langle x^4,x^2y^2,y^4\rangle$. Let $J=I+(x^3y,y^3x,x^4)$, the algebra $B=S/J$ has a symmetric Hasse-Hilbert diagram but is not Gorenstein.  \brick{Moreover,} the Hasse-Hilbert diagram has as greatest element $B_{(4,1)}$, that is, $B=\oplus_{g \leq (4,1)}B_g$.

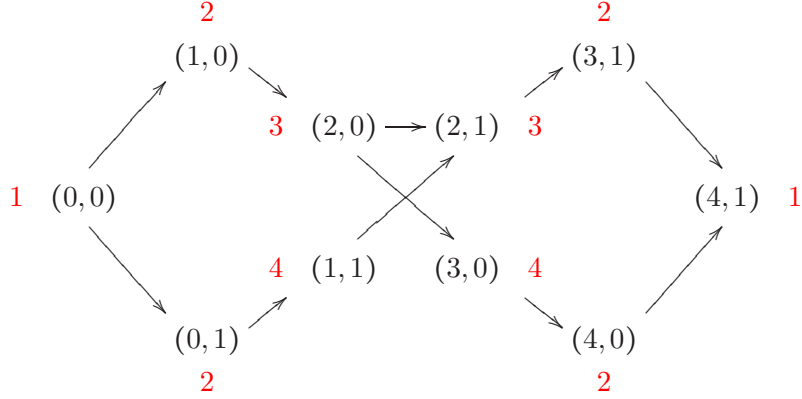
$$
\xymatrix@=.4em{ 
& & {\begin{matrix}{\color{red}2}\\(1,0)\end{matrix}} \ar[rd]& & & &   {\begin{matrix}{\color{red}2}\\(3,1)\end{matrix}}\ar[rrdd]  & \\
& & &{\begin{matrix}{\color{red}3}&(2,0)\end{matrix}}\ar[rrdd]\ar[rr] & &{\begin{matrix}(2,1)&{\color{red}3}\end{matrix}}\ar[ru] & \\
{\begin{matrix}{\color{red}1}&(0,0)\end{matrix}}\ar[rruu]\ar[rrdd]&&&&&&&&{\begin{matrix}(4,1)&{\color{red}1}\end{matrix}}\\
&  & & {\begin{matrix}{\color{red}4}&(1,1)\end{matrix}}\ar[rruu]& &{\begin{matrix}(3,0)&{\color{red}4}\end{matrix}}\ar[rd] &&\\
& &{\begin{matrix}(0,1)\\{\color{red}2}\end{matrix}}\ar[ru] &&&&{\begin{matrix}(4,0)\\{\color{red}2}\end{matrix}}\ar[rruu]&
}
$$

\end{ex}

\subsection{On the Codimension One conjecture}

\label{A geometric case for Cox-Gorenstein algebras}

In this subsection, we provide a geometric example for Cox-Gorenstein algebras. It comes as a particular case of the Codimension One conjecture, raised by Cattani, Cox, and Dickenstein. Here we consider $\mathbb K=\mathbb C$.

\begin{conj}{\rm(\cite[Conjecture 3.12]{1997})}
If $X$ is a complete simplicial toric variety and $F_0, \ldots , F_n \in B(\Sigma)$ are homogeneous polynomials that do not vanish simultaneously in $X$, then:
\[\dim_{\C} (S_{\omega}/(F_0, \ldots , F_n)_{\omega}) = 1\] 
where $\omega=\sum_i\deg F_i-\beta_0$ is the critical degree of $F_0, \ldots , F_n$ and $\beta_0$ the anticanonical class of $X$.
\end{conj}

\paragraph{Our special case.}
We prove a particular case of the above conjecture. To this  purpose, we discuss some generalities about the Koszul and \v{C}ech complexes (see \cite{BS} for some basics about these topics).

Let $R$ be a $G$-graded ring, $\mathbf f=f_1,\ldots,f_r\in R$ be homogeneous elements with $g_i:=\deg f_i\in G$ for all $i=1,\ldots, r$, and $I=(\mathbf f)$. Also, for every $j=1\ldots, r$, write
$$\textbf{X}_j=\{g_{i_1}+\cdots+g_{i_j}\mid i_1<\cdots<i_j\}.$$

We consider the local \v{C}ech complex $C^\bullet$ of $R$ with respect to the sequence $\mathbf f$ and the Koszul complex $K_\bullet$ of $R$ with respect to $\mathbf f$. Recall that $C^\bullet$ computes local cohomology $H^i_I(R)$, see \cite[Chapter 5]{BS}. Both complexes are $G$-graded, and in particular note that for all $j=1,\ldots,r$,
$$K_j=\bigoplus_{h\in\textbf{X}_j}R(-h)^{b_{j,h}}$$
where $b_{j,h}$ is a positive integer. The first quadrant double complex $C^\bullet(K_\bullet)\otimes_RM$ (with $C^0(K_0)\otimes_RM$ centered at the origin)
$$\xymatrix@=1em{
& \vdots\ar[d] & \vdots\ar[d] & \vdots\ar[d]
\\
0\ar[r] & C^0(K_1)\otimes_RM\ar[d]\ar[r] & C^1(K_1)\otimes_RM\ar[r]\ar[d] & C^2(K_1)\otimes_RM\ar[r]\ar[d] & \dots
\\
0\ar[r] & C^0(K_0)\otimes_RM\ar[r]\ar[d] & C^1(K_0)\otimes_RM\ar[r]\ar[d] & C^2(K_0)\otimes_RM\ar[r]\ar[d] & \dots
\\
& 0 & 0 & 0
}$$
gives rise to two spectral sequences $E$ and $'E$ by taking, respectively, horizontal and vertical homology first. Also, both $E$ and $'E$ converge to the same $\mathbb Z$-graded $R$-module $H$ (considering the trivial $\mathbb Z$-grading on $R$).

Note that
$$E_1^{i,j}=K_j(\mathbf f;H^i_I(M))\simeq H^i_I(K_j(\mathbf f;M))\simeq\bigoplus_{h\in\textbf{X}_j}H^i_I(M)^{b_{j,h}}(-h)$$
and $E_2^{i,j}=H_j(\mathbf f; H^i_I(M)).$ On the other hand, the spectral sequence induced by taking homologies in the vertical first is such that
$$'E_1^{i,j}=C^i(H_j(\mathbf f;M))$$
and thus $'E_1^{i,j}=0$ for all $i>0$ and $'E_1^{0,j}=H_j(\mathbf f;M)$ for all $j\geq0$, as $I=(\mathbf f)$ annihilates $H_j(\mathbf f;M)$ and $C^i$ makes the $f_i's$ invertible for $i>0$. It follows from convergence that
$$H_j\simeq\begin{cases} {}'E_1^{0,j}\simeq H_j(\mathbf f;M), & \forall j\geq0 \\ H_j=0, & \forall j<0.\end{cases}$$ Eventually, we conclude that there exists a spectral sequence
$$E_2^{i,j}=H_j(\mathbf f; H^i_I(M))\Rightarrow_j H_{j-i}(\mathbf f;M)$$
with $E_1^{i,j}=H^i_I(K_j(\mathbf f;M))\simeq\bigoplus_{h\in\textbf{X}_j}H^i_I(M)^{b_{j,h}}(-h)$.

Furthermore, for any $l\in G$ we can take the $l$-th component in the above spectral sequence
\begin{align}\label{spectralsequence}
E_2^{i.j}(l)=H_j(\mathbf f;H^i_I(M))_l\Rightarrow_j H_{j-i}(\mathbf f; M)_l
\end{align}
with $E_1^{i,j}(l)=\oplus_{h\in\textbf{X}_j}H^i_I(M)^{b_{j,h}}_{l-h}$.

\begin{prop}\label{koszulcech}
Let $M$ be a $G$-graded $R$-module. Suppose that there are $g\in G$ and $s=1,\ldots, r$ such that
$$\bigoplus_{j=1}^r\bigoplus_{h\in\textbf{X}_j}H^i_I(M)_{g-h}=0, \forall i\neq s.$$
Then
$$H_t(\mathbf f;M)_g\simeq H_{t+s}(\mathbf f; H^s_I(M))_g, \forall t\geq0.$$

In particular:
\begin{enumerate}
    \item if either $H^s_I(M)=0$ or $i>r-s$, then $H_i(\mathbf f; M)_g=0$.

    \item $H_{r-s}(\mathbf f;M)_g\simeq H_r(\mathbf f; H^s_I(M))_g$. Particularly, if $\bigoplus_i H^s_I(M)_{g-\sum_{j\neq i}g_j}=0$ then $H_{r-s}(\mathbf f; M)_g\simeq H^s_I(M)_{g-\sum_i g_i}$.

    \item $H^s_I(M)_g=[IH^s_I(M)]_g$.
    
\end{enumerate}
\end{prop}

\begin{proof}
Items (1) and (2) follow immediately from the first isomorphism. To prove such an isomorphism and item (3), by taking $l=g$ in the spectral sequence (\ref{spectralsequence}) we see that $E_1^{i,j}(g)=0$ for all $i\neq s$ and for all $j\geq0$. It follows by convergence that
$$H_i(\mathbf f;M)_g\simeq E_2^{s,i+s}(g)\simeq H_{i+s}(\mathbf f; H^s_I(M))_g, \forall i\geq0.$$
Also, $E_2^{s,0}(g)=H_{-s}(\mathbf f;M)_g=0$, i.e., $[H^s_I(M)/IH^s_I(M)]_g=H_0(\mathbf f; H^s_I(M))_g=0$.
 \end{proof}

\begin{cor}\label{r=s}
By taking $s=r$, $H_i(\mathbf f;M)_g=0$ for $i>0$ and $[M/IM]_g\simeq H_r(\mathbf f; H^r_I(M))_g$.
\end{cor}

The previous results hold for any abstract grading $G$. Now we go to an application in toric geometry;   the field $\K$ will be the field $\C$ of complex numbers, and the grading group will be the class group $\Cl(\Sigma)$ of a toric variety.

Recall that a Weil divisor $D=\Sigma_{\rho} a_\rho D_\rho$ gives the polyhedron
$$P_D=\{m\in M_{\mathbb R}\mid \langle m,u_\rho \rangle \geq-a_\rho, \forall \rho\in\Sigma(1)\}$$
which is a polytope when $\Sigma$ is complete, see \cite[Chapter 9]{CoxLittleSchenck}.

\begin{thm}\label{picard1}
Let ${\Pj_\Sigma}$ be a $d$-dimensional complete simplicial toric variety with Cox ring $S$, and let $f_i\in S_{\alpha_i}$ be homogeneous polynomials with $i=0,\ldots,d$, and all $\alpha_i$ nef and the polytope $P_\eta$ is full dimensional for all $\eta\in\bigcup_j\textbf{X}_j$. Let us assume that the $f_i's$ do not vanish simultaneously on $\Pj_\Sigma$. If $A=S/(f_0,\ldots,f_d)$ then
\begin{enumerate} \itemsep=-2pt
\item $A_\omega\simeq\mathbb C$, where $\omega=\sum_{i=0}^d\alpha_i-\beta_0$ and $\beta_0$ stands for the anticanonical class.
\item $x_\rho A_{\omega}=0$ where $x_\rho$ is the variable in $S$ corresponding to the ray $\rho$.
\end{enumerate}

Moreover, if the Picard number of $\Pj_\Sigma$ is $1$, then $(f_0,\ldots,f_d)$ is a complete intersection, and $A$ is a Cox-Gorenstein $\mathbb C$-algebra with socle degree $\omega$.
\end{thm}

\begin{proof}
The hypothesis that the $f_i$'s do not vanish simultaneously on $\Pj_\Sigma$ means that, by the toric Nullstellensatz, $\sqrt{(\mathbf f)}=B(\Sigma)$, where $\mathbf f=f_0,\ldots,f_d$ and $B(\Sigma)$ stands for the irrelevant ideal of $S$. Thus, the local \v{C}ech complex $C^\bullet$ of $S$ with respect to the sequence $\mathbf{f}$ computes local cohomology supported on the ideal $B(\Sigma)$, see \cite[1.2.3 Remark]{BS}.

Now, note that $H^0_{B(\Sigma)}(S)=H^1_{B(\Sigma)}(S)=0$ and $H^i_{B(\Sigma)}(S)_\beta\simeq H^{i-1}(\Pj_\Sigma,\mathcal O_{\Pj_\Sigma}(\beta))$ for all $i>1$ and $\beta\in \Cl(\Sigma)$ (\cite[Exercise 9.5.6 and Theorem 9.5.7]{CoxLittleSchenck}). Also, given
$$\gamma\in\{\omega+[D] \mid [D] \ \mbox{is one of the generators of} \ \Cl(\Sigma)\}\cup\{\alpha_N\}$$
and $\eta\in\textbf{X}_j$ where $j=1,\ldots, d+1$, we have $$H^i_{B(\Sigma)}(S)_{\gamma-\eta}\simeq H^{d-i+1}(\Pj_\Sigma,\mathcal O_{\Pj_\Sigma}(\eta-\gamma-\beta_0))^\vee$$
by Serre duality. Since $\eta-\gamma-\beta_0$ is nef and its polytope is full dimensional for any $\eta$, we conclude that $H^i_{B(\Sigma)}(S)_{\gamma-\eta}=0$ for all $i=2,\ldots,d$ by the Batyrev-Borisov vanishing \cite[Theorem 9.2.7]{CoxLittleSchenck}. Thus, we are able to apply Corollary \ref{r=s} for $r=d+1$ and doing so we have $H_i(\mathbf f;S)_\gamma=0$ for all $i>0$ and
$$A_\gamma\simeq H_{d+1}(\mathbf f;H^{d+1}_{B(\Sigma)}(S))_\gamma=\ker(H^{d+1}_{B(\Sigma)}(S)_{\gamma-\sum_i\alpha_i}\to\oplus_{j=1}^{d+1}H^{d+1}_{B(\Sigma)}(S)^{b_{d,j}}_{\gamma-\sum_{j\neq i}\alpha_j})$$ 
for all $\gamma\in\{\omega+[D]: [D] \ \mbox{is one of the generators of} \ \Cl(\Sigma)\}\cup\{\omega\} $. By applying Serre duality again we conclude that $H^{d+1}_{B(\Sigma)}(S)_{\gamma-\sum_i\alpha_i}=0$ for $\gamma=\omega+[D]$ and $H^{d+1}_{B(\Sigma)}(S)_{\gamma-\sum_i\alpha_i}\simeq\mathbb C$ for $\gamma=\omega$. Also, for $\gamma=\omega$ and for any $j=1,\ldots, d+1$, $H^{d+1}_{B(\Sigma)}(S)_{\gamma-\sum_{j\neq i}\alpha_i}\simeq H^0(\Pj_\Sigma,\mathcal O_{\Pj_\Sigma}(-\alpha_j))^\vee=0$ by the Batyrev-Borisov vanishing again. It follows that
$A_{\omega}\simeq\mathbb C$, whence item 1, and $A_{\omega+[D]}=0$ for any generator $[D]$ of $\Cl(\Sigma)$, which means the statement in item 2.

Suppose now that the Picard number of $\Pj_\Sigma$ is $1$. It follows that $\sqrt{I}={B(\Sigma)}=\mathfrak{m}$, where $I=(\mathbf f)$ and $\mathfrak{m}$ is the ideal generated by the variables of $S$, and thus $A$ is an Artinian local $\mathbb C$-algebra with $\Cl(\Sigma)$-grading, where $\Cl(\Sigma)$ is the Class group of $\Pj_\Sigma$. Since $\sqrt{I}=\mathfrak{m}$, we have that $H^i_I(S)\simeq H^i_\mathfrak{m}(S)$ for all $i\geq0$ by \cite[1.2.3 Remark]{BS} so that $H^i_I(S)=0$ for all $i\neq d+1$ by the graded local duality \cite[14.4.1 Theorem]{BS}, which means by \cite[6.2.7 Theorem]{BS} that $I$ is a complete intersection (i.e., the $f_i$'s form a regular sequence in $S$, see \cite[Corollary 1.6.19]{BH}). In particular, we conclude that $A=S/(\mathbf f)$ must be Cox-Gorenstein.

Finally, from item 2 one sees that $A_{\omega}\subseteq\soc(A)$ and since $A$ is Gorenstein, we actually have $A_{\omega}=\soc(A)$, i.e., $\omega$ is the socle degree of $A$.
\end{proof}

As a corollary of the first part of Theorem \ref{picard1}, we derive a case of the Codimension One Conjecture. As we remarked in the Introduction, this result is already contained in \cite[Corollary 2.5]{Cox-Dick}.

\begin{cor}\label{almost}
\blue{Let $\Pj_\Sigma$ be a complete simplicial toric variety with   Cox ring $S$, and let  $f_0,\ldots,f_d\in B(\Sigma)$ be homogeneous polynomials which do not vanish simultaneously on $\Pj_\Sigma$.} If $\deg f_i$ is nef for all $i=0,\ldots,d$ and the polytopes $P_\eta$ are full dimensional, where $\eta\in\bigcup_j\mathbf{X}_j$, then
$$\dim_{\mathbb C}(S_{\omega}/(\mathbf f)_{\omega})=1$$
where $\omega=\sum_i\deg f_i-\beta_0$, where $\beta_0$ stands for the anticanonical class.
\end{cor}

The assumption on the Picard number \blue{in the second part of  Theorem \ref{picard1}} is essential, as we can see in the next example.
\begin{ex}\label{counterexample}\rm(Example \ref{ex1}) Let $\Pj_\Sigma$ be $\Pj^1\times \Pj^1$ with coordinates $x,y,u,v$ and degrees $\deg(x)=\deg(y)=(1,0)$ and $\deg(u)=\deg(v)=(0,1)$. The ideal $J =(x^2u-y^2v,x^2v,y^2u)\subset\C[x,y,u,v]$ is such that $\mathbb{V}(J)=\emptyset$ and $R=S/J$ cannot be an Artinian $\Z^2$-graded algebra  since $J$ has height $2$.

 \end{ex}

\subsection{Toric Macaulay-Matlis duality}\label{Macaulay_Matlis}

Let $G$ be an Abelian group and let $S=\K[x_1, \ldots, x_s]$ be the polynomial ring with a 
$G$-grading. Let $Q=\K[X_1, \ldots, X_s]$ be the ring of differential operators associated to $S$, that is, $X_i=\frac{\partial}{\partial x_i}$, and $S$ has a natural structure of $Q$-module given by differentiation $X_i(x_j)=\delta_{ij}$. The $G$-grading on $S$ induces a $G$-grading on $Q$ by defining $\deg(X_i)=\deg(x_i) \in G$. Notice that the action has a descending degree, that is, for  $g \preceq h $ we have $Q_{g} \times S_{h} \to S_{ h - g }$.    

Given a $G$-graded $Q$-submodule $M$ of $S$, we define the annihilator \[\Ann(M) = \{\alpha \in Q|\ \alpha(f) = 0\ \forall f \in M\}.\]
Just as in the classical case $\Ann(M) \subset Q$ is an ideal and it is easy to see that it is $G$-homogeneous. Conversely, given a $G$-homogeneous ideal $I\subset Q$ we define the inverse system
\[I^{-1} = \{f  \in S|\ \alpha(f) = 0\ \forall \alpha \in I\}.\]
Again, $I^{-1} \subset S$ is a $G$-graded $Q$-submodule of $S$.

If  $\deg(\alpha)=\deg(f)=a$ then $\alpha \cdot f\in \K$. Hence we have a perfect pairing 
\[Q_{a}\times S_{a}\rightarrow \K , \quad (\phi,f)\mapsto \phi\cdot f . \]
 One immediate consequence is that $\dim (Q/I)_g=\dim (I^{-1})_g$. Summarizing, we get the $G$-graded differential version of Macaulay-Matlis duality.

\begin{thm}\label{thm:doubleMacaulay} We have a bijective correspondence
$$
\begin{array}{ccc}
 \{\text{$G$-homogeneous ideals}\,\ I\subset Q\}&  \leftrightarrow  & \{ \text{$G$-graded Q-submodules of}\,\ S\} \\
 I& \mapsto  & I^{-1}\\
 Ann(M) & \mapsfrom & M
\end{array}  
$$
Under this correspondence $M=I^{-1}$ is finitely generated as an $S$-module if and only if $A=Q/I$ is Artinian. Moreover, $A$ is Artinian Gorenstein if and only if $M=Q\cdot f$ is a cyclic module.   
\end{thm}
\begin{proof} A more general fact can be found in \cite[Theorem 21.6]{Eisenbud} where the key point is the canonical module and its relation with the dualizing functor for local Artinian rings.
    The last part is the so-called double annihilator theorem of Macaulay. 
\end{proof}

\begin{thm} Let $I \subset Q =\K[X_1, \ldots, X_s]$ be a $G$-homogeneous ideal such that $A=Q/I$ is an Artinian $G$-graded $\K$-algebra. 
Then $A$ is Cox-Gorenstein of socle degree $\omega$ if and only if there is $f\in S_{\omega}$ such that $I=\operatorname{Ann}(f)$.
\end{thm}

\begin{rmkk} Let $A=Q/I$ be a Cox-Gorenstein $G$-graded $\K$-algebra with $Q=\K[X_1,\ldots,X_s]$ and $I= \Ann(f)\subset Q$. Notice that the socle degree $\omega$ coincides with the degree of $f$; it is the socle degree of $A$.  We can assume that the presentation $A=Q/I$ is minimal and call the embedding dimension the codimension of $A$.
\end{rmkk}

\begin{rmkk}
    We recall that if $\{\alpha_1, \ldots, \alpha_s\} \subset Q$ is a maximal regular sequence, then the ideal $I=(\alpha_1, \ldots, \alpha_s) \subset Q$ is an Artinian complete intersection, it is well known that Artinian complete intersections are Gorenstein. Moreover, if $\alpha_i$ are $G$-homogeneous elements, then $I$ is $G$-homogeneous and $A=Q/I$ is a Cox-Gorenstein algebra.   
\end{rmkk}

\begin{ex} Let $G=\Z \oplus \Z_2$ and consider $Q = \K[X,Y,Z]$ $G$-graded by $\deg(X)=(1,\overline{1})$, $\deg(Y)=(1,\overline{0})$ and $\deg(Z)=(2,\overline{1})$. We are going to consider the toric Jacobian ideal\footnote{For the notion of toric Jacobian ideal see \cite{1997}.}  of the $G$-homogeneous polynomial $F = X^4+Y^4+Z^2$ with $\deg(F)=(4, \overline{0})$. Let $I=J_0(F)=(X\frac{\partial F}{\partial X}, Y\frac{\partial F}{\partial Y}, Z\frac{\partial F}{\partial Z})=(X^4,Y^4,Z^2)$. Since $(X^4,Y^4,Z^2)\subset Q$ is a complete intersection, we know that $A=Q/I$ is a Cox-Gorenstein $G$-graded $\K$-algebra of socle degree $(8,\overline{0})$. Considering $Q$ as a ring of differential operators, we get $I=\Ann(f)$ with $f=x^3y^3z\in \K[x,y,z]$.

In fact, when $\K$ is the field of complex numbers $\C$, $Q$ is the Cox ring of a fake weighted projective space, see Example \ref{fake2}.  
\end{ex}

\subsection{Artinian and Gorenstein minimal reductions}\label{minimalreductions}

In this subsection, we explain how a maximal element in a  $G$-graded algebra gives rise to Artinian Gorenstein quotients, and how this construction connects the primitive cohomology of a hypersurface (\ref{primitive}) with Cox–Gorenstein algebras.

 Let $(G,\preceq)$ be a partially ordered group and let $R$ be a $G$-graded $\K$-algebra having a maximal graded piece $R_{\omega}\neq 0$, that is, for all $g\in G$ such that $\omega \preceq g$ and $\omega\neq g$ we have $R_g=0$. We can consider a natural Artinian quotient that preserves the Hasse-Hilbert diagram for all $g \preceq \omega$ and has $\omega$ as the greatest degree. 

 \begin{ex}\label{ex4}  
Let $S=\C[x,y,u,v]$ be $\Z^2$-graded by $\deg(x)=\deg(y)=(1,0)$ and $\deg(u)=\deg(v)=(0,1)$, as in Example \ref{ex1}. Let $I =(x^2u-y^2v,x^2v,y^2u) \subset S$. In Example \ref{counterexample}, we have seen that $R=S/I$ is not an Artinian $\Z^2$-graded algebra. On the other hand, the Hasse-Hilbert diagram has a maximal element $\omega = (4,1)$ by Theorem \ref{picard1}. 
Let $J = I + S_{(5,0)} + S_{(0,2)}$, and let $A=S/J \simeq R/\overline{J}$. Then $A$ is an Artinian algebra, for all $g \leq \omega$ we have $A_g\simeq R_g$, and $A$ has the greatest element in degree $\omega = (4,1)$.  
In fact $A=\oplus_{g \leq (4,1)}A_g$ (see Example \ref{ex1}).  
\end{ex}

The algorithm is very natural; consider $J=I+\sum S_h$ for $h$ incomparable with $\omega$. Of course, we can restrict ourselves to minimal incomparable degrees.  We get the following result. 

\begin{prop}  \label{prop:artinianization} Let $A=S/I$ be a $G$-graded $\K$-algebra having a maximal graded piece $A_{\omega}\neq 0$.
    Then there exists a unique minimal $G$-homogeneous ideal $J$ containing $I$ such that  $S/J$ is an Artinian graded $\K$-algebra having $\omega$ as the greatest element. Moreover, by the minimality of $J$ we have $(S/J)_{g}\simeq (S/I)_{g}$, 
    for all  $g \preceq \omega$.
    
\end{prop}

\begin{ex}

Assuming that $A$ is an Artinian $G$-graded algebra having unidimensional greatest element $A_{\omega} \simeq \K$, we get pairings: 
\[A_g\times A_{\omega-g} \to \K.\]
For all $g \preceq \omega.$ As we saw in Example \ref{ex1}, this pairing is not necessarily perfect. The reason for this phenomenon is the existence of socle in other degrees. For instance in Example \ref{ex1} one can easily check that $\dim_{\K} (\soc(A)\cap A_{(4,0)})=3$. Moduling out the socle outside the maximal degree we obtain a Cox-Gorenstein algebra.   

$$
\xymatrix@=.4em{ 
& & {\begin{matrix}{\color{red}2}\\(1,0)\end{matrix}} \ar[rd]& & & &   {\begin{matrix}{\color{red}2}\\(3,1)\end{matrix}}\ar[rrdd]  & \\
& & &{\begin{matrix}{\color{red}3}&(2,0)\end{matrix}}\ar[rrdd]\ar[rr] & &{\begin{matrix}(2,1)&{\color{red}3}\end{matrix}}\ar[ru] & \\
{\begin{matrix}{\color{red}1}&(0,0)\end{matrix}}\ar[rruu]\ar[rrdd]&&&&&&&&{\begin{matrix}(4,1)&{\color{red}1}\end{matrix}}\\
&  & & {\begin{matrix}{\color{red}4}&(1,1)\end{matrix}}\ar[rruu]& &{\begin{matrix}(3,0)&{\color{red}4}\end{matrix}}\ar[rd] &&\\
& &{\begin{matrix}(0,1)\\{\color{red}2}\end{matrix}}\ar[ru] &&&&{\begin{matrix}(4,0)\\{\color{red}2}\end{matrix}}\ar[rruu]&
}
$$

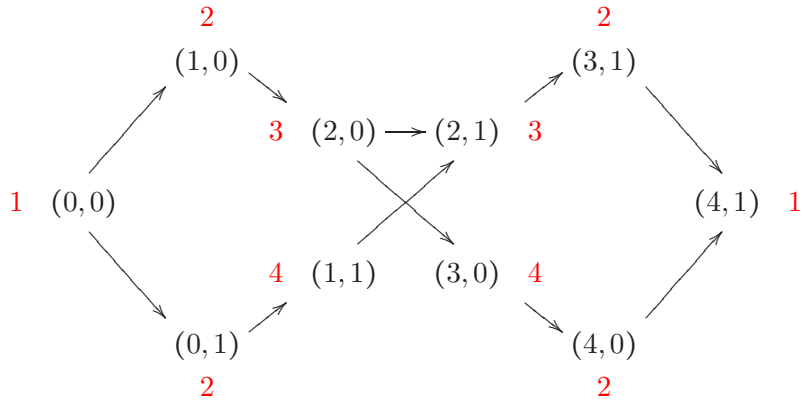
\centering{Hasse-Hilbert diagram of the minimal Gorenstein quotient}

\end{ex}

\begin{prop} \label{prop:gorenstenianization} Let $A=S/I$ be an Artinian $G$-graded $\K$-algebra having a unidimensional greatest element $A_{\omega}$. 
There exists a unique minimal $J$ containing $I$ such that $B=S/J$ is a Cox-Gorenstein graded $\K$-algebra.
    Moreover, the quotient is minimal in the sense that it preserves the pairings that were already perfect. 
\end{prop}

\begin{rmk}
When $A$ arises from a projective simplicial toric variety, there are cohomology vanishing conditions in order to check  if some (all) pairings are perfect, see \cite[Theorem 2.1]{Villaflor} for details.

\end{rmk}

 Until the end of this Section we take the complex numbers $\C$ for the ground field $\K$.

\begin{dfn}\label{primitive} Let $X_f$ be an hypersurface in $\Pj_{\Sigma}$ cut by $f\in B(\Sigma)_{\beta}$. The primitive cohomology group $H_{\rm prim}^{d-1}(X)$ is the quotient $H^{d-1}(X_f,\C)/i^{*}H^{d-1}(\Pj_{\Sigma},\C)$, where $i$ is the inclusion map $X_f\hookrightarrow \Pj_{\Sigma}$ .
    
\end{dfn}

\begin{dfn} Given the ideal $J_0(f)=(x_1\partial f/\partial x_1,\dots, x_n\partial f/\partial x_n )\subset S$, we define the quotient ideal $J_1(f):=(J_0(f):x_1\cdots x_n)$ and  we denote $R_1(f):=S/J_1(f)$.
    
\end{dfn}

 It is known that $R_1(f)$ has a unidimensional greatest element when $f$ is general and the primitive cohomology $H_{\rm prim}^{d-1}(X)$ of an ample hypersurface has a natural ${\rm Cl}(\Sigma)$-grading. Moreover, in the case of an ample general hypersurface with non-trivial primitive cohomology, we have the following result.

\begin{cor} Let $X_f$ be a general ample hypersurface with a non-trivial primitive class; then $H_{\rm prim}^{d-1}(X_f)$ embeds canonically into the Cox-Gorenstein algebra associated to the Artinian and Gorenstein minimal reduction of $R_1(f)$.
    
\end{cor}

\begin{proof} It follows by Proposition 4.15 in \cite{BatyrevCox}, Corollary 3.1 in \cite{Villaflor} and Proposition 4.17.
    
\end{proof}

For more details about the primitive cohomology of a hypersurface, see \cite[Section 3.2]{BruzzoGrassi}.

\section{Toric Lefschetz properties and the Toric Hessian criteria}
\label{lefschetz properties}

\subsection{Toric Lefschetz properties}

Let $A = Q/I=\oplus_{g \in G}  A_{g}$ be an Artinian $G$-graded $\K$-algebra. Let $\phi \in \operatorname{Hom}_\Z(G,\Q)$ and let $\mathcal{L}=\langle X_1,\ldots,X_s\rangle\subseteq A$ be the $\K$-linear subspace generated by the class of the variables in $Q$. We say that any homogeneous element $L \in \mathcal{L}$ is linear.  \brick{ Let $\mathcal{L}_g = \mathcal{L}\cap A_g$.}

We say that two graded pieces of $A$, say $A_{g}$ and $A_{h}$, are 
\begin{itemize}\itemsep=-2pt \item linearly consecutive if \brick{$g$ and $h$ are consecutive,} and $\mathcal{L}_{h-g}\neq 0$;
\item  linearly comparable if $g \preceq h$ and there is $L \in \mathcal{L}_{l}$ such that $h=g+kl$ for some $k \in \Z_+$. 
\end{itemize}

$$
\xymatrix@=.4em{ 
& & {\begin{matrix}{\color{red}2}\\(1,0)\end{matrix}}\ar@[blue]@<.5ex>@{-}[rd] \ar@<-.5ex>[rd]& & & &   {\begin{matrix}{\color{red}2}\\(3,1)\end{matrix}}\ar@[blue]@{-}@<.5ex>[rrdd]\ar@<-.5ex>[rrdd]  & \\
& & &{\begin{matrix}{\color{red}3}&(2,0)\end{matrix}}\ar@[blue]@<.5ex>@{-}[rr]\ar@<-.5ex>[rrdd]\ar@<.5ex>@[blue]@{-}[rrdd]\ar@<-.5ex>[rr] & &{\begin{matrix}(2,1)&{\color{red}3}\end{matrix}}\ar@[blue]@{-}@<.5ex>[ru]\ar@<-.5ex>[ru] & \\
{\begin{matrix}{\color{red}1}&(0,0)\end{matrix}}\ar@{-}@[blue]@<-.5ex>[rrdd]\ar@<-.5ex>[rruu]\ar@<.5ex>[rrdd]\ar@{-}@[blue]@<.5ex>[rruu]&&&&&&&&{\begin{matrix}(4,1)&{\color{red}1}\end{matrix}}\\
&  & & {\begin{matrix}{\color{red}4}&(1,1)\end{matrix}}\ar@<.5ex>[rruu]\ar@<-.5ex>@[blue]@{-}[rruu]& &{\begin{matrix}(3,0)&{\color{red}4}\end{matrix}}\ar@[blue]@{-}@<.5ex>[rd]\ar@<-.5ex>[rd] &&\\
& &{\begin{matrix}(0,1)\\{\color{red}2}\end{matrix}}\ar@[blue]@<-.5ex>@{-}[ru]\ar@<.5ex>[ru] &&&&{\begin{matrix}(4,0)\\{\color{red}2}\end{matrix}}\ar@[blue]@{-}@<-.5ex>[rruu]\ar@<.5ex>[rruu]&
}
$$

\centerline{Hasse-Hilbert diagram and linear comparability}

\begin{dfn} 1. We say that $A$ has the Toric Weak Lefschetz property (TWLP) if for every linearly consecutive graded pieces $A_{g}$ and $A_{h}$ there is a linear element $L \in \mathcal{L}_{h-g} $ such that the $\K$-linear multiplication map $\bullet L: A_{g} \to A_{h}$ has maximal rank. 

2. We say that $A$ has the Toric Strong Lefschetz property (TSLP) if for every linearly comparable pieces $A_{g}$ and $A_{h}$ there is a linear element $L \in \mathcal{L}_{l} $ with $h=g+kl$ such that the $\K$-linear multiplication map $\bullet L^k: A_{g} \to A_{h}$ has maximal rank. 
\end{dfn}

\begin{rmk}
    Notice that the previous definitions generalize the WLP and the SLP in the standard graded case where each variable has the same degree. 
    
\end{rmk}

\begin{ex} (See Example \ref{ex1}).  Consider $S=\K[x,y,u,v]$ and $G=\Z^2$ and a $G$-grading given by $\deg(x)=\deg(y)=(1,0)$ and 
$\deg(u)=\deg(v)=(0,1)$. Let $f \in S_{(2,3)}$ be given by $f=x^2u^3+y^2v^3$. Let $Q=\K[X,Y,U,V]$ be the ring of differential operators acting on $S$ and let $I=\Ann(f)\in Q$ be the Artinian Cox-Gorenstein ideal producing 
$A=Q/I=\oplus_{k=0}^5 A_k $ and $A_k = \oplus_{i+j=k} A_{ij}$. Here we are considering the total degree as a $\Z$-grading for $A$.   In this case, we have two spaces of linear elements, $\mathcal{L}_{(1,0)}=\langle X,Y\rangle$ and $\mathcal{L}_{(0,1)}=\langle U,V\rangle$. By the standard $\Z$-grading, we know that the WLP can be verified in the middle, that is, from $A_2$ to $A_3$. Since $A_2=A_{(2,0)}\oplus A_{(1,1)}\oplus A_{(0,2)}$ and $A_3=A_{(2,1)}\oplus A_{(1,2)}\oplus A_{(0,3)}$  we see that $L=U+V$ can be used as linear element simultaneously for $\bullet L: A_{(2,0)} \to A_{(2,1)}$ , $\bullet L: A_{(1,1)} \to A_{(1,2)}$  and $\bullet L: A_{(0,2)} \to A_{(0,3)}$. We know that $A_{(2,0)}=\langle X^2,Y^2\rangle$,  $A_{(2,1)}=\langle X^2U,Y^2V\rangle$ and the multiplication map by $L$ is actually an isomorphism. 
In the same way we can verify that $\bullet L: A_{(1,1)} \to A_{(1,2)}$  and $\bullet L: A_{(0,2)} \to A_{(0,3)}$ are isomorphisms and $A$ has the WLP. On the other side, $A_{(2,0)}$ and $A_{(0,3)}$ are symmetric with respect to Poincaré duality but they are not linearly comparable.

$$
\xymatrix@=.4em{ 
&& &&  {\begin{matrix}{\color{red}2} \\ (2,0)\end{matrix}}\ar@{~}[rrrrr]\ar@<-.5ex>[rrrrrdddddddddd]\ar@[blue]@{-}@<.5ex>[rrrrrdddddddddd] &&&&& {\begin{matrix}{\color{red}2} \\ (0,3)\end{matrix}}\ar@<-.5ex>[rrddd]\ar@[blue]@{-}@<.5ex>[rrddd]
\\ \\ \\
&& {\color{red}2} \ (1,0)\ar@[blue]@<.5ex>@{-}[rrdd]\ar@<-.5ex>[rrdd]\ar@<-.5ex>[rruuu]\ar@[blue]@<.5ex>@{-}[rruuu] &&&&&&&&& (1,3) \ {\color{red}2}\ar@<-.5ex>[rrdd]\ar@[blue]@{-}@<.5ex>[rrdd]
\\ \\
{\color{red}1} \ (0,0)\ar@<.5ex>[rrdd]\ar@[blue]@<-.5ex>@{-}[rrdd]\ar@<-.5ex>[rruu]\ar@[blue]@{-}@<.5ex>[rruu] && && {\color{red}2} \ (1,1)\ar@<-.5ex>[rrrrrddddd]\ar@[blue]@<.5ex>@{-}[rrrrrddddd]\ar@<-.5ex>[rrrrr]\ar@[blue]@{-}@<.5ex>[rrrrr] &&&&& (1,2) \ {\color{red}2}\ar@<-.5ex>[rruu]\ar@[blue]@{-}@<.5ex>[rruu]\ar@<.5ex>[rrdd]\ar@[blue]@{-}@<-.5ex>[rrdd] &&&& (2,3) \ {\color{red}1}
\\ \\
&& {\color{red}2} \ (0,1)\ar@<.5ex>[rruu]\ar@[blue]@<-.5ex>@{-}[rruu]\ar@[blue]@<-.5ex>@{-}[rrddd]\ar@<.5ex>[rrddd] &&&&&&&&& (2,2) \ {\color{red}2}\ar@<.5ex>[rruu]\ar@[blue]@{-}@<-.5ex>[rruu]
\\ \\ \\
&& && {\begin{matrix}(0,2) \\ {\color{red}2}\end{matrix}}\ar@<.5ex>[rrrrruuuuuuuuuu]\ar@[blue]@<-.5ex>@{-}[rrrrruuuuuuuuuu]\ar@<.5ex>[rrrrruuuuu]\ar@[blue]@{-}@<-.5ex>[rrrrruuuuu]\ar@{~}[rrrrr] &&&&& {\begin{matrix}(2,1) \\ {\color{red}2}\end{matrix}}\ar@<.5ex>[rruuu]\ar@[blue]@{-}@<-.5ex>[rruuu]
}
$$

\centerline{Hasse-Hilbert diagram and linear comparability}

\end{ex}

\begin{ex}
    Let $S=\K[x,y,z]$ be $\Z$-graded by $\deg(x)=\deg(y)=1$ and $\deg(z)=2$. Let $f \in S_4$ given by $f=x^4+y^4+z^2$. 
    In the dual $Q=\K[X,Y,Z]$ we obtain $\Ann(f)=(XY,XZ,YZ,X^5,Y^5,Z^3,X^4-Y^4,X^4-Z^2)$. Let $A=Q/I$ be the Cox-Gorenstein algebra associated with $f$. We have $A=A_0\oplus A_1\oplus A_2\oplus A_3\oplus A_4 $ with $A_1=\langle X,Y\rangle$, $A_2=\langle X^2,Y^2,Z\rangle$, $A_3=\langle X^3,Y^3\rangle$, and $A_4=\langle X^4\rangle= \langle Y^4 \rangle =\langle Z^2\rangle$. It is easy to verify that $A$ has the TSLP with the linear element $L=X+Y$. 
    
\end{ex}

\subsection{Toric differential Euler identity}

\brick{The following Proposition is quite straightforward, and was inspired by \cite[Lemma 3.8]{BatyrevCox} (see also \cite{KleimanKleppe})}.

\begin{prop} Suppose that $S$ is $G$-graded with $\deg(x_i)=g_i$ and let $f \in S_g$.
If $\phi\in\operatorname{Hom}_{\Z}(G,\Q)$, then there exists a generalized Euler relation

\[\sum_{i=1}^s\phi(g_i)x_i\frac{\partial f}{\partial x_i}=\phi(g)\cdot f .
\]  
\end{prop}

\begin{proof}
    By linearity it is enough to consider the monomial case, that is $f=x_1^{e_1}x_2^{e_2}\ldots x_s^{e_s}$.
    By hypothesis $\deg(f) = \sum_{i=1}^se_ig_i=g \in G$.
It is easy to see that $x_i\frac{\partial f}{\partial x_i} = e_if$, therefore $\phi(g_i)x_i\frac{\partial f}{\partial x_i} = \phi(e_ig_i)f$. Summing every term we get 
\[ \displaystyle \sum_{i=1}^s \phi(g_i)x_i\frac{\partial f}{\partial x_i} = f  \displaystyle \sum_{i=1}^s \phi(e_ig_i) = f\phi(\sum_{i=1}^se_ig_i)=\phi(g)\cdot f.\]
    
\end{proof}

The following proposition was inspired by the classical differential Euler identity, see, for example, \cite[Lemma 7.2.19]{russo2016}.

\begin{dfn}  \brick{Let $\phi\in\Hom_\Z(G,\Q)$. We say that $L \in \mathcal{L}_g$ is $\phi$-linear if $\phi(\deg(L))=1$.}
    
\end{dfn}

\begin{prop}[Toric differential Euler identity]  Let $L=a_1X_1+\ldots+a_mX_m \in \mathcal{L}_g$ be a $\phi$-linear element  and $f$ a homogeneous polynomial of degree $\omega\in G$ such that $\phi(\omega)\in \Z_+$. Then,

$$L^{\phi(\omega)}f=\phi(\omega)! f(a) \ \ \ \ where \ \ \ \ a=(a_1,\dots,a_m,0,\dots,0). $$
    
\end{prop}

\begin{proof} 
\begin{multline} L^{\phi(\omega)}f=L^{\phi(\omega)-1}\circ L f= L^{\phi(\omega)-1} \left( \sum_{i=1}^ma_i\frac{\partial f}{\partial x_i}\right)= \\  \sum_{i=1}^ma_i\left(L^{\phi(\omega)-1}\frac{\partial f}{\partial x_i}\right)=  \sum_{i=1}^m a_i\left(\phi(\omega)-1 \right)!\frac{\partial f}{\partial x_i}(a).
\end{multline}
The last equality holds by induction on $\phi(\omega)$ and using the Euler identity and keeping in mind that $L$ is linear we have $L^{\phi(\omega)}f=\phi(\omega)! f(a) $.
\end{proof}

\subsection{Toric Hessian Criterion}
In this subsection we present natural extensions of definitions and results in \cite[Section 2]{GondimZappala'2018}.

\begin{dfn} Let $\mathcal{B}=\{\beta_1,\dots, \beta_s\}$ and $\mathcal{C}=\{\gamma_1,\dots,\gamma_t\}$ be $\K$-basis of $A_{g}$ and $A_{g'}$ respectively. The toric mixed Hessian of $f$  with mixed order $(\beta,\gamma) $ is
$$\Hess^{(\mathcal{B},\mathcal{C})}_f:=\left[\beta_i\circ \gamma_j(f) \right] $$
     
\end{dfn}

\begin{rmk}
    Let $\mathcal{C}^{*}$ the dual basis of $\mathcal{C}$ in the sense of the Poincar\'e duality, i.e., there exists an isomorphism $\varphi: A_{\omega-\gamma}\to \Hom_\K(A_{\gamma},\K)$ defined by $\varphi(\delta) \tau=\delta.\tau$  and

$$
\gamma^{*}_i\circ \gamma_j=\left\lbrace \begin{array}{cc}
  \Omega   & i=j \\
    0 & i\neq j 
\end{array} \right.
 $$   
where $\Omega\in A_{\omega}$ is the generator of the socle such that $\Omega(f)=1$.

\end{rmk}

We denote by $\Hess^{(\mathcal{C}^{\ast},\mathcal{B})}_f$ the mixed Hessian with respect to the dual basis of $\mathcal{C}$ and $\mathcal{B}$.

\begin{thm}[Toric Hessian Criterion]
Let $A=Q/I$ with $I=\Ann(f)$ be an Artinian Cox-Gorenstein $G$-graded $\K$-algebra. Let $A_{g}$ and $A_{h}$ be two linearly comparable graded pieces of $A$ such that $h=g+kl$, and let $L=a_1X_1+\ldots+a_mX_m\in \mathcal{L}_l$ be a $\phi$-linear element such that $\phi(\deg(f))\in \Z_{+}$. Then the matrix of the $\K$-linear map $\bullet L^k:A_{g}\to A_{h}$ with respect to bases $\mathcal{B}$ and $\mathcal{C}$ can be given by:

$$\left[ \cdot L^k\right]_{\mathcal{B}}^{\mathcal{C}}=k!\cdot  \Hess^{(\mathcal{C}^*,\mathcal{B})}_f (a) 
\quad\text{where}\quad  a=(a_1,\dots, a_m,0,\dots,0) .$$
\end{thm}

\begin{proof} Let $\mathcal{B}=\{\beta_1,\ldots,\beta_s\}, \mathcal{C}=\{\gamma_1,\ldots,\gamma_t\}$, and $\left[ \cdot L^k\right]_{\mathcal{B}}^{\mathcal{C}}=[\mu_{ij}]$. Then $L^k\beta_j=\sum _{l=1}^t \mu_{lj} \gamma_l \in A$. Multiplying by $\gamma_i^*$ the previous equality and evaluating in $f$ we get
$\gamma_i^*(L^k\beta_j)f=\sum _{l=1}^t \mu_{lj} \gamma_i^* \gamma_l\eq \mu_{ij}\Omega f= \mu_{ij}$. Let $f_{ij}=\gamma_i^*\beta_j f\in S_{kg}$. Since $\phi(\deg(f_{ij})) = k\in \Z_+$, by the differential Euler identity, we have
$$\mu_{ij}=L^k\gamma_i^*\beta_j f=L^kf_{ij}= k!f_{ij}(a) $$
and whence
$$\left[ \cdot L^k\right]_{\mathcal{B}}^{\mathcal{C}}=k!\cdot  \Hess^{(\mathcal{C}^*,\mathcal{B})}_f (a). $$

\end{proof}

\begin{rmk} \brick{\cite[Lemma 2.5]{AbdallahMcdaniel} is a special case of the previous result.}
\end{rmk}
\begin{ex} Let $G =\langle g\rangle$ be a cyclic ordered Abelian group. Let $g_1,\ldots,g_n \in G$ with $g_i=d_ig$, $d_i\in \Z$. 
Let $S=\C[x_1, \ldots, x_n]$ with grading $\deg(x_i) = g_i$. 
Let $d=\prod d_i$ and $k_i=\prod_{j\neq i} d_j = d/d_i$. Define $f=\sum_{i=1}^n x_i^{k_i}$ the Fermat form of degree $dg$. Consider the algebra $A=Q/\Ann(f)$. 
It is easy to see that $A_{kg} =\langle X_i^{e_i}\rangle$ such that $e_id_i=k$ with $i=1,\ldots,n$. Therefore it has the TSLP by the Toric Hessian criterion. In fact, the mixed Hessian matrix after a choice of ordering is lower triangular with powers of the variables in the diagonal.  
\end{ex}

\section{Appendix: Another view on the  Cox-Gorenstein ideals}

We give a characterization of {\em Cox-Gorenstein ideal}  which generalizes the definition given by Otwinowska in \cite{Otwinowska2003ComposantesDP} for homogeneous ideals in the coordinate ring of a projective space.

\begin{prop}  Let $R$ be a $G$-graded $\K$-algebra such that $R_0\simeq\K$, and let $I$ be a homogeneous ideal such that
$A=R/I$ is Artinian. Then $I$ is Cox-Gorenstein   of socle degree  $\omega\in G$ if and only if there exists a linear functional $\Lambda \in (R_{\omega})^\vee$ such that
\begin{equation}\label{Lambda} I_g = \{ P\in R_g \,\vert\, \Lambda(PQ) =0 \ \  \text{for all} \  \ Q\in R_{\omega-g}\}.\end{equation}
\end{prop}

\begin{proof} 
If $I$ is a Cox-Gorenstein ideal, the functional $\Lambda\in (R_{\omega})^\vee$ defined by the composition $R_{\omega} \to A_{\omega} \to \K$ satisfies condition \eqref{Lambda}. Indeed, \eqref{Lambda} certainly holds if $P\in I_g$. On other hand, if \eqref{Lambda} holds for some $P\in R_g$ then it implies $u\cdot v =0$
for all $v\in A_{\omega-g}$ and $u=\bar P$, so that $u=0$ and $P\in I_g$.

{\em Vice versa,} such a $\Lambda$
defines a pairing $\tilde{\Lambda} :  A_g\times A_{\omega-g}\to \mathbb{K}$. If $u\in A_g$ is such that $\tilde\Lambda(u,v)=0$ for all $v\in A_{\omega-g}$, then $\Lambda(P,Q)=0$ for all $Q\in R_{\omega-g}$
if $\bar P =u$, so that $P\in I_g$, i.e., $u=0$, and $\tilde\Lambda$ is non-degenerate, and $A_{\omega-g}\simeq A_g^\vee$.

Note that this implies $A_{\omega}\simeq \K$ and $ A_{\omega-g}=0$ if $A_g=0$.
\end{proof}

The following proposition generalizes a result in \cite{Otwinowska2003ComposantesDP}.

\begin{prop}\label{lambda}Let $I\subset J$ be Cox-Gorenstein ideals with socle degree $\omega'$ and $\omega$ respectively. Let $\Lambda'$, $\Lambda$ be the functionals associated with $I$ and $J$. Then there exists $F\in R_{\omega'-\omega}$ such that $\Lambda(Q)=\Lambda'(QF)$ for all $Q\in R_{\omega}$.
\label{propACG}
\end{prop}

\begin{proof} Dualizing the exact sequence 
$$ 0 \to I_{\omega} \to S_{\omega} \to S_{\omega}/I_{\omega} \to 0$$
we obtain
$$ 0 \to (S_{\omega}/I_{\omega})^\vee \to  (S_{\omega})^\vee \to  I_{\omega}^\vee \to 0.$$
 Since $\Lambda$ is zero on $ I_{\omega}$ it lies in $(S_{\omega}/I_{\omega})^\vee$. The latter is isomorphic to $S_{\omega'-\omega}/I_{\omega'-\omega}$ via $\Lambda'$, so that there is a polynomial $F\in S_{\omega'-\omega}$ satisfying the required condition.
\end{proof}

\begin{prop}\label{F}  Let $I'$ be a Cox-Gorenstein ideal with socle degree $\omega'$;   for any  $0\neq F\notin I'$ with $\deg(F)=\omega'-\omega$ the ideal $I=(I':F)$ is a Cox-Gorenstein of socle degree $\omega$.
    \end{prop}
    \begin{proof} If $\Lambda'\in (S_{\omega'})^\vee$ is the functional associated to $I'$, define $\Lambda\in (S_{\omega})^\vee$ as $\Lambda(Q) = \Lambda'(QF)$. It is clear that $\Lambda(PQ)=0$ if
    $P\in I_g$ and $Q\in S_{\omega-g}$. On the other hand, if
    $P\in S_g$ and $\Lambda(PQ)=\Lambda'(PFQ)=0$ for all $Q\in S_{\omega-g}$ then $PF \in (I')_{\omega'-\omega+g}$ so that $P\in I_g$.
    \end{proof}

\end{document}